\documentclass[12pt,a4paper]{amsart}

\usepackage{amsmath,amsfonts,amssymb,amsthm,mathrsfs}

\usepackage{color}

\newcommand{\R}{\mathbb{R}}

\newcommand{\mn}{|\!\!|}
 
\newcommand{\be}{\begin{equation}} 
\newcommand{\ee}{\end{equation}}
\newcommand{\bea}{\begin{eqnarray}} 
\newcommand{\eea}{\end{eqnarray}}
\newcommand{\bean}{\begin{eqnarray*}} 
\newcommand{\eean}{\end{eqnarray*}}
\newcommand{\rf}[1]{(\ref {#1})}

\def\dx{\,{\rm d}x}
\def\dy{\,{\rm d}y}

\def\ds{\,{\rm d}s}

\def\e{\varepsilon}

\def\xn{|\!|\!|}

\newtheorem{theorem}{Theorem}

\newtheorem{lemma}[theorem]{Lemma}

\theoremstyle{definition}

\def\qed{\hfill $\square$}

\theoremstyle{remark}
\newtheorem{remark}[theorem]{Remark}

\def\int{\intop\limits}

\numberwithin{equation}{section}
\numberwithin{theorem}{section}

\author[P. Biler]{Piotr Biler}
\address{\small Instytut Matematyczny, Uniwersytet Wroc\l awski,
 pl. Grunwaldzki 2/4, 50-384 Wroc\-\l aw, Poland}
\email{Piotr.Biler@math.uni.wroc.pl}

\author[T. Cie\'slak]{Tomasz Cie\'{s}lak}
\address{\small Instytut Matematyczny, Polska Akademia Nauk, ul. \'Sniadeckich 8, 00-956 Warsaw, Poland}
\email{T.Cieslak@impan.pl}

\author[G. Karch]{Grzegorz Karch}
\address{\small 
 Instytut Matematyczny, Uniwersytet Wroc\l awski,
 pl. Grunwaldzki 2/4, 50-384 Wroc\-\l aw, Poland}
\email{Grzegorz.Karch@math.uni.wroc.pl}

\author[J. Zienkiewicz]{Jacek Zienkiewicz}
\address{\small 
 Instytut Matematyczny, Uniwersytet Wroc\l awski,
 pl. Grunwaldzki 2/4, 50-384 Wroc\-\l aw, Poland}
\email{Jacek.Zienkiewicz@math.uni.wroc.pl}

\title[2D chemotaxis models]{Local criteria for blowup \\ in two-dimensional chemotaxis models}

\begin{document}

\begin{abstract} 
We consider two-dimensional versions of the Keller--Segel model for the  chemotaxis with either classical (Brownian) or fractional (anomalous) diffusion. Criteria for blowup of solutions in terms of suitable Morrey spaces norms are derived. Moreover, the impact of the consumption term on the global-in-time existence of solutions is analyzed for the classical Keller--Segel system. 
\end{abstract}



\date{\today}

\thanks{    This work is accepted for publication in DCDS-A }

\maketitle

\baselineskip=17pt

\section{Introduction}
We consider in this paper the following version of the parabolic-elliptic  Keller--Segel model of chemotaxis in two space dimensions
\bea
u_t+(-\Delta)^{\alpha/2}u+\nabla\cdot(u\nabla v)&=&0,\ \ x\in {\mathbb R}^2,\ t>0,\label{equ}\\ 
\Delta v-\gamma v+u &=& 0,\ \    x\in {\mathbb R}^2,\ t>0,\label{eqv}
\eea
supplemented with the initial condition 
\be
u(x,0)=u_0(x)\label{ini}.
\ee
Here the unknown variables $u=u(x,t)$ and $v=v(x,t)$ denote the density of the population and the density of the chemical secreted by the microorganisms, respectively, and  the given consumption (or degradation) rate of the chemical is denoted by $\gamma\ge 0$. The  diffusion operator is described either by the usual Laplacian ($\alpha=2$) or by a fractional power of the Laplacian $(-\Delta)^{\alpha/2}$ with  $\alpha\in(0,2)$. The initial data are nonnegative functions $u_0\in L^1(\mathbb R^2)$ of the total mass 
\be 
M=\int u_0(x)\dx.\label{M}
\ee

Our main results include criteria for blowup of nonnegative solutions of problem \rf{equ}--\rf{ini} expressed in terms of a local concentration of data (Theorem \ref{blow}), and  the existence of global-in-time solutions for the initial condition of an arbitrary mass $M$ and each sufficiently large $\gamma$ (Theorem \ref{ex}). The novelty of these blowup results consists in using local properties of solutions instead of a comparison of the total mass and moments  of a solution as was done in {\it e.g.} \cite{N}, \cite{KS-JEE},   \cite{LR07,LR08,LRZ}, \cite{BW}, \cite{BKL}, \cite{KS-AM}, and \cite{BK-JEE}. In particular, we complement the result in \cite{KS-JEE} saying that solutions of \rf{equ}--\rf{ini} with $\alpha=2$, fixed $\gamma\ge 0$ and sufficiently well concentrated $u_0$ with  $M>8\pi$ blow up in a finite time, by showing that solutions of that system with $u_0$ of arbitrary $M>0$ and all sufficiently large $\gamma$ are global-in-time. 

Many previous works have dealt with the existence of global-in-time solutions with small data in critical Morrey spaces, i.e.~those which are scale-invariant under a natural scaling of the chemotaxis model, cf. {\it e.g.} \cite{B-SM} and \cite{Lem}. Our criteria for a  blowup of solutions  with large concentration can be expressed by  Morrey space norms (see Remark \ref{r3} below for more details), and we have found  that the size of such a norm is critical for the global-in-time existence versus finite time blowup. The analogous question for radially symmetric solutions of the  $d$-dimensional Keller--Segel model with $d\ge 3$ has been recently studied in  \cite{BKZ}.

%%%%%%%%%%%%%%%%%%%%%%%%%%
\section{Statement of results}

It is well-known that problem \rf{equ}--\rf{ini} with $\alpha=2$ has a unique 
mild solution $u\in {\mathcal C}([0,T); L^1(\R^2))$ 
for every $u_0\in L^1(\R^2)$ and $\gamma\geq 0$.
Here, as usual, a mild solution satisfies a suitable integral formulation \rf{D} of the Cauchy problem  \rf{equ}--\rf{ini} as recalled at the beginning of Section 5.  Moreover, given $u=u(x,t)$, we define $v=(-\Delta +\gamma)^{-1}u$, see Lemma \ref{H}, below. 
Due to a parabolic regularization effect
(following {\it e.g.}  \cite[Th. 4.2]{GMO}), this solution is smooth for $t>0$, hence, it satisfies the Cauchy problem in the classical sense. Moreover, it conserves the total mass \rf{M}
$$
\int_{\R^2} u(x,t)\dx=\int_{\R^2}u_0(x)\dx\qquad \text{for all $t\in [0,T)$},
$$
and is nonnegative when $u_0\geq 0$. Proofs of these classical results can be found {\it e.g.} in \cite{KS-JEE, Lem, K-O, KS-AM,BDP}, %,BM}, 
see also Section 5 of this work.
Analogous results on local-in-time solutions to the Cauchy problem  \rf{equ}--\rf{ini} with $\alpha\in (1,2)$ have been obtained in \cite{BK-JEE}, \cite[Th. 2]{Lem}. 
To the best of our knowledge, \cite[Th. 1.1]{LRZ} and a recent \cite[Th. 1, Th. 2]{SYK}  are the only results on local-in-time classical solutions of the Cauchy problem \rf{equ}--\rf{ini} with $\alpha\in (0,1]$, $d\ge 2$. 
Thus, the case (iii) of Theorem \ref{blow} asserts
that such a solution cannot be global-in-time for initial data satisfying \rf{a:cond}.

In our first result,
we formulate  new sufficient conditions for blowup 
({i.e.} nonexistence for all $t>0$)
of such local-in-time solutions of problem \rf{equ}--\rf{eqv}.

\begin{theorem}\label{blow}
Consider $u\in {\mathcal C}([0,T); L^1(\R^2))$ --- a local-in-time nonnegative  classical solution  of   problem \rf{equ}--\rf{ini} with a nonnegative $u_0\in  L^1(\R^2)$.  

\noindent
(i) If $\alpha=2$, $\gamma=0$ (the scaling invariant Keller--Segel model), then for each $M>8\pi$ the solution $u$ blows up in a finite time. 

\noindent 
(ii) Let $\alpha=2$ and $\gamma>0$ (the Keller--Segel model with the consumption).
 If $M>8\pi$ and if $u_0$ is well concentrated   around a point $x_0\in \R^2$, namely, there exists $R>0$ such that  
\be
{\rm e}^{-\sqrt{\gamma}R}\int_{\{|y-x_0|<R\}}u_0(y)\dy>8\pi\ \ {\rm and\ \ } \int_{\{|y-x_0|\ge R\}}u_0(y)\dy<\nu 
\label{ass:blow}
\ee
with  an explicitly computed small constant $\nu>0$, then the solution $u$ blows up in a finite time.

\noindent
(iii) Let  $\alpha\in(0,2)$ and $\gamma\ge 0$ (the Keller--Segel model with fractional diffusion). If there exist $x_0\in\R^2$ and  $R>0$ such that 
\begin{equation}\label{a:cond} 
 R^{\alpha-2} \int_{\{|y-x_0|<R\}} u_0(y)\dy>C \ \ {\rm and\ \ }\int_{\{|y-x_0|\ge R\}}u_0(y)\dy<\nu, 
\end{equation}
for some explicit constants: small $\nu>0$ and big $C>0$,  then the solution $u$ ceases to exists in a finite time.
\end{theorem}

\begin{remark} \label{r2}
The result (i)  for $\alpha=2$ and $\gamma=0$ is, of course, well known, but the proof below slightly differs from the previous ones. 
The case (ii) $\alpha=2$ and $\gamma>0$ has been considered in \cite{KS-JEE} but the sufficient conditions for blowup were expressed in terms of {\em globally} defined quantities: i.e. mass $M>8\pi$ and the moment $\int  u_0(x)|x|^2 \dx$.
\end{remark}

\begin{remark} \label{r3}
The case (iii) $\alpha<2$. 
Recall that the (homogeneous) Morrey   space $M^p(\R^2)$ is defined as the space of locally integrable functions such that 
$$
\mn u\mn_{M^p}=\sup_{R>0,\, x\in\R^2}R^{2(1/p-1)}\int_{\{|y-x|<R\}}u(y)\dy<\infty.
$$
The first condition in \eqref{a:cond} is equivalent to a sufficiently large Morrey norm of $u_0$ in the space $M^{2/\alpha}(\R^2)$. 
Indeed, obviously we have 
$$\mn u_0\mn_{M^{2/\alpha}}\ge
 R^{\alpha-2} \int_{\{|y-x_0|<R\}} u_0(y)\dy$$ for every $x_0$ and $R>0$, but also there is 
 $x_0\in\R^2$ and $R>0$ such that $$\mn u_0\mn_{M^{2/\alpha}}\le
2  R^{\alpha-2} \int_{\{|y-x_0|<R\}} u_0(y)\dy.$$
Thus, our blowup  condition in terms of the Morrey norm seems to be new and complementary to that guaranteeing the global-in-time existence of solutions, where  smallness of initial conditions in  the $M^{2/\alpha}$-Morrey norm has to be imposed, cf.  prototypes of such results in \cite[Theorem 1]{B-SM} and \cite[Remark 2.7]{BK-JEE}. 
\end{remark}

\begin{remark} \label{r4}
 A natural scaling for  system \rf{equ}--\rf{eqv} with $\gamma=0$:
$$
u_\lambda(x,t)=\lambda^\alpha u(\lambda x,\lambda^\alpha t),$$
leads to the equality $\int u_\lambda \dx = \lambda^{\alpha-2}\int u \dx$, i.e. 
mass  of rescaled solution $u_\lambda$ can be chosen arbitrarily  with a suitable $\lambda>0$.  Thus, the conditions in Theorem 
\ref{blow}.iii are insensitive to the actual value of $M$, so w.l.o.g. we may suppose  that $M=1$. 
\end{remark}

\begin{remark} \label{r5}
 The second parts of the condition  \rf{ass:blow} and \rf{a:cond} are not scaling invariant. However, we believe that these assumptions are not necessary for the conclusion in Theorem \ref{blow}.ii, \ref{blow}.iii. In fact, one can prove it for $\alpha$ close to 2 by an inspection of methods in \cite{BKZ,BZ}.
\end{remark}

Next, we show that 
the first condition in  the concentration assumptions \rf{ass:blow} is in some sense optimal to obtain a blowup of solutions. In the following theorem we show that for every initial integrable function $u_0$, even with its $L^1$-norm  above $8\pi$,  the corresponding   mild solution to the model \rf{equ}--\rf{ini} with $\alpha=2$ is global-in-time for all sufficiently large consumption rates $\gamma>0$. 

\begin{theorem}\label{ex}
Let $\alpha=2$, $\gamma>0$. 
For each $u_0\in L^1(\mathbb R^2)$, there exists $\gamma(u_0)>0$ such that for all $\gamma\ge\gamma(u_0)$ the Cauchy problem \rf{equ}--\rf{ini} has a global-in-time mild solution satisfying $u\in {\mathcal C}([0,\infty);L^1(\mathbb R^2))$. This is a classical solution of system \rf{equ}--\rf{eqv}  for $t>0$, and satisfies  for each $p\in[1,\infty)$ the decay estimates 
\be
\sup_{t>0}t^{1-1/p}\|u(t)\|_p<\infty.\label{esti}
\ee 
\end{theorem}

Thus,  for each $u_0$ (not necessarily nonnegative) and $\gamma$ large enough depending on  $u_0\in {L^1(\mathbb R^2)}$,  solutions of the Cauchy problem are global-in-time, so there is no critical value of mass which leads to a blowup of solutions. 
On the other hand, if $M>8\pi$, then for $0\le\gamma\ll 1$ the solutions blow up in a finite time, as it is seen from the sufficient conditions for blowup in Theorem \ref{blow}.ii.

\section{Notation and preliminaries} 

In the sequel, $\|\cdot\|_q$ denotes the usual $L^q(\mathbb R^2)$ norm, and $C$'s are generic constants independent of $t$, $u$, $z$, ...   which may, however, vary from line to line. 
 Integrals with no integration limits are meant to be calculated over the whole plane. 

Let us denote by $G$ the Gauss--Weierstrass kernel of the heat semigroup ${\rm e}^{t\Delta}$ on $L^p({\mathbb R}^2)$ 
space 
\be
G(x,t)=(4\pi t)^{-1}\exp\left(-\frac{|x|^2}{4t}\right).\label{GW}
\ee 
As it is well known the convolution with $G$, denoted by $G(t)\ast z={\rm e}^{t\Delta}z$, satisfies the following $L^q-L^p$ estimates 
\be
\|{\rm e}^{t\Delta}z\|_p\le Ct^{1/p-1/q}\|z\|_q\label{lin1}
\ee
and 
\be 
\|\nabla{\rm e}^{t\Delta}z\|_p\le Ct^{-1/2+1/p-1/q}\|z\|_q\label{lin2}
\ee
 for all $1\le q\le p\le\infty$, $t>0$. 
  Moreover, for each $p>1$ and $z\in L^1(\R^2)$ the following relation holds 
 \be
 \lim_{t\to 0}t^{1-1/p}\|{\rm e}^{t\Delta}z\|_p=0\label{function}
 \ee
 which is, {\it e.g.}, noted in \cite[Lemma 4.4]{GMO}.

\begin{lemma} \label{H}
For every $\gamma>0$,
the operator $(-\Delta+\gamma)^{-1}$  solving the Helmholtz equation \rf{eqv} satisfies 
 \be
\|\nabla (-\Delta+\gamma)^{-1} z\|_q\le C\gamma^{1/p-1/q-1/2}\|z\|_p,\label{Kgamma}
\ee 
 for every $1\le p<2< q<\infty$ such that  $\frac1p-\frac1q< \frac12$ and some $C$ independent of $\gamma$. In the critical case $\frac1p-\frac1q= \frac12$ inequality \eqref{Kgamma} also holds provided $p>1$. 
 Moreover,  the Bessel kernel $K_\gamma$ of $(-\Delta+\gamma)^{-1}$ has the following pointwise behavior at $0$ and $\infty$ 
\bea
\nabla K_\gamma(x)&\sim&-\frac{1}{2\pi}\frac{x}{|x|^2}\ \  {\rm as}\ \ x\to 0,\label{0}\\ 
|\nabla K_\gamma(x)|&\le& C\frac{1}{|x|}{\rm e}^{-\sqrt{\gamma}|x|}\ \ {\rm as} \ \ x\to\infty,\label{infty}
\eea
and satisfies the global one-sided bound 
\be
x\cdot\nabla K_\gamma(x)\le -\frac{1}{2\pi}{\rm e}^{-\sqrt{\gamma}|x|}.\label{point}
\ee
 \end{lemma} 
 \medskip 

\proof 
The proof of inequality \rf{Kgamma} requires separate arguments in two cases, $\frac1p-\frac1q< \frac12$ and $\frac1p-\frac1q= \frac12$. In the first case, the result is a consequence of inequalities \rf{lin1} and \rf{lin2} by representing the operator  $(-\Delta+\gamma)^{-1}$ as the Laplace transform 
\be 
(-\Delta+\gamma)^{-1}=\int_0^\infty {\rm e}^{-\gamma s}{\rm e}^{s\Delta}\ds.\label{K}
\ee
Indeed, we have  the following representation of $K_\gamma$ in the Fourier variables
$$
\widehat{(K_\gamma\ast z)}(\xi)=\frac{1}{|\xi|^2+\gamma}\hat z(\xi) =\int_0^\infty {\rm e}^{-\gamma s}{\rm e}^{-s|\xi|^2}\hat z(\xi)\ds,
$$
so that 
\bea
\|\nabla (-\Delta+\gamma)^{-1} z\|_q&\le& C\int_0^\infty {\rm e}^{-\gamma s}s^{1/q-1/p-1/2}\ds\,  \|z\|_p\nonumber\\
&\le& C\gamma^{1/p-1/q-1/2}\int_0^\infty {\rm e}^{-s}s^{-1/2+1/q-1/p}\ds\, \|z\|_p,\nonumber
\eea 
the latter integral is finite due to the assumption on $p$ and $q$. 

When  $\frac1p-\frac1q= \frac12$,  inequality \rf{Kgamma}  follows from the end-point case of the Sobolev inequality $\left\|\nabla (-\Delta)^{-1}u\right\|_q\leq C\left\|u\right\|_p$. 

\noindent
For properties \rf{0}, \rf{infty} and \rf{point}, see  {\it e.g.} \cite[Lemma 3.1]{KS-JEE} and \cite[Ch. V, Sec. 6.5]{S}.
 \qed

\begin{remark} \label{r6}
Let us note that the reference \cite[Theorem 2.9]{KS-AM} provides us with precise conditions on radial convolution kernels $K$ leading to a blowup of solutions of general diffusive aggregation equations with the Brownian diffusion of the form 
$u_t-\Delta u+\nabla\cdot(u(\nabla K\ast u))=0$. They are strongly singular, i.e. they have the singularity at $0$: 
$\limsup_{x\to 0} x\cdot\nabla K (x)<0$, and are of moderate growth at $\infty$:
$| x\cdot\nabla K (x)|\le C|x|^2$.  Of course, the Bessel kernel $K_\gamma$ is strongly singular in the sense of \cite{KS-AM}, as it is seen from \rf{point}. 
\end{remark}

\section{Blowup of solutions } 

In this section we prove Theorem \ref{blow} using the method of truncated moments which is reminiscent of that in the papers \cite{N1}, \cite{K-O}.  
First, we define the ``bump'' function $\psi$ and its rescalings for $R>0$ 
\be
\psi(x)=(1-|x|^2)_+^2
=\left\{
\begin{array}{ccc}
(1-|x|^2)^2& \text{for}& |x|<1,\\
0& \text{for}& |x|\geq 1,
\end{array}
\right.
\qquad  \psi_R(x)=\psi\bigg(\frac{x}{R}\bigg). \label{bump}
\ee
The function $\psi$ is piecewise ${\mathcal C}^2(\mathbb R^2)$, with ${\rm supp}\,\psi=\{|x|\le 1\}$,   and  satisfies 
\bea
\nabla\psi(x)&=&-4x(1-|x|^2)\ \ {\rm for\ }\ |x|<1, \label{gradpsi}\\
\Delta\psi(x)&=&(-8+16|x|^2) \ge -8\psi(x)\ge -8\,\ {\rm for\ \ }{|x|<1}.\label{lappsi} 
\eea

We will use in the sequel the %following
 fact that $\psi$ is strictly concave  in a neighbourhood of $x=0$.

\begin{lemma}\label{prop1}
For each $\e\in\left(0,\frac{1}{\sqrt{3}}\right)$, the function $\psi$ defined in \eqref{bump} is strictly concave for all $|x|\le \e$. More precisely, $\psi$  satisfies
\begin{equation}\label{0.1}
H\psi\leq -\theta(\e) I
\end{equation}  
for all $|x|\le \e$, 
where $H\psi$ is the Hessian matrix of second derivatives of $\psi$, 
$\theta(\e) =4\left(1-3\e^2\right)$, and $I$ is the identity matrix.  
In particular, we have 
\begin{equation}\label{wazne}
\theta(\e)\nearrow 4\quad \mbox{as}\quad  \e\searrow 0.
\end{equation}
\end{lemma}
\proof
For every $\xi\in \R^2$ the following identity  holds
\[
\xi\cdot H\psi\, \xi= 4\left(-|\xi|^2\left(1-|x|^2\right)+2(x\cdot\xi)^2\right).
\]
Thus, by the Schwarz inequality, we have
$\xi\cdot H\psi\, \xi \le 4|\xi|^2\left(3|x|^2-1\right)$.  \qed

\medskip

Next, we recall a well-known property of concave functions. 

\begin{lemma}\label{prop2}
For every function $\Psi:\R^2\rightarrow \R$ which is strictly concave on a domain $\Omega\subset\R^2$ we have for all $x,y\in \R^2$ 
\begin{equation}\label{concave}
(x-y)\cdot\left(\nabla \Psi(x)-\nabla \Psi(y)\right)\leq -\theta |x-y|^2, 
\end{equation}
where $\theta>0$ is the constant of strict concavity of $\Psi$ on $\Omega$, i.e. satisfying  $H\Psi\le -\theta \,I$.
\end{lemma}

\proof
By the concavity, we obtain
\[
\Psi(x)\leq \Psi(y)+\nabla\Psi(y)\cdot(x-y)-\frac{\theta}{2!}|x-y|^2.
\] 
Summing this inequality with its symmetrized version (with $x,\,y$ interchanged) leads to the claim. \qed

\medskip

We have the following scaling property of the fractional Laplacian  
\be
(-\Delta)^{\alpha/2}\psi_R(x)=R^{-\alpha}\big((-\Delta)^{\alpha/2}\psi\big)_R,\label{scal}
\ee  
and we notice the following boundedness property of $(-\Delta)^{\alpha/2}\psi$. 
\medskip

\begin{lemma}\label{Getoor}
For every $\alpha\in(0,2]$ there exists a constant $k_\alpha>0$ such that   
\be
\left|(-\Delta)^{\alpha/2}\psi(x)\right|\le k_\alpha.\label{deltapsi}
\ee 
Moreover, $(-\Delta)^{\alpha/2}\psi(x)\le 0$ for $|x|\ge 1$. 
In particular,  for $\alpha=2$ we have $k_2=8$.   
\end{lemma}

\proof 
For $\alpha=2$, this is an obvious consequence of the explicit form of $\psi$, hence we assume $\alpha\in(0,2)$. 

To show estimate \eqref{deltapsi} for $\alpha\in (0,2)$, 
it suffices to use the following well-known representation of the fractional Laplacian with $\alpha\in(0,2)$ 
\be 
 (-\Delta)^{\alpha/2}\psi(x)=-c_\alpha\, {\rm P.V.} \int \frac{\psi(x+y)-\psi(x)}{|y|^{2+\alpha}}\dy \nonumber \\
\ee 
 for certain explicit constant $c_\alpha>0$. 
Now, using the Taylor formula together with 
the fact that $\psi,\, D^2\psi\in L^\infty(\R^2)$, we immediately obtain that the integral on the right-hand side is finite and uniformly bounded in $x\in\R^2$. 
Since $\psi(x)\ge 0$ and $\psi(x)=0$ for $|x|\ge 1$ we have 
$$
(-\Delta)^{\alpha/2}\psi(x)=-c_\alpha\, {\rm P.V.}\int \frac{\psi(x+y)}{|y|^{2+\alpha}}\dy\le 0
$$
for $|x|\ge 1$.      
\qed

\bigskip

Now, we formulate a crucial inequality in our proof of the blowup result. 

\begin{lemma}\label{prop3}
For the Bessel kernel $K_\gamma$ with $\gamma\geq 0$ and a strictly concave function  $\Psi$  we have for all $x,y$ on the domain of the strict concavity of $\Psi$ 
\begin{equation}\label{1.1}
\nabla K_\gamma (x-y)\cdot \left(\nabla \Psi(x)-\nabla \Psi(y)\right)\geq \frac{\theta}{2\pi}g_{\gamma}(|x-y|),
\end{equation}
where $\theta$ is the constant of the strict concavity of $\Psi$ introduced in Lemma \ref{prop2}, and $g_\gamma$ is a radially symmetric continuous function,   such that
\begin{equation}\label{bessel}
\nabla K_\gamma(x)=-\frac{1}{2\pi}\frac{x}{|x|^2}g_\gamma(|x|).
\end{equation}
In particular, $g_\gamma(0)=1$, the profile of $g_\gamma$ decreases, 
and $g_\gamma(|x|)\le C{\rm e}^{-\sqrt{\gamma}|x|}$. 
\end{lemma}
\proof
Combining Lemma \ref{prop2} with equation \eqref{bessel}
and properties \rf{0}, \rf{infty} and \rf{point}
 we arrive immediately  at the claimed formula.  \qed 
\medskip

We are in a position to prove our main blowup result. 
\medskip

\noindent {\bf Proof of Theorem \ref{blow}.} We consider the quantity
$$w_R(t)=\int u(x,t)\psi_R(x)\dx,$$ 
a {\em local} moment of $u(.,t)$, where $\psi_R(x)$ is 
defined in \rf{bump} for each $R>0$. 
Let 
\be\label{Mt}
M_R(t)\equiv \int_{\{|x|<R\}}u(x,t)\dx \geq w_R(t)
\ee
 denote mass of the distribution $u$  contained in the ball $\{|x|<R\}$ at the moment $t$.
Now,  using equation \rf{equ} we determine the evolution of $w_R(t)$ 
\bea
\frac{\rm d}{{\rm d}t}w_R(t)
&=& -\int(-\Delta)^{\alpha/2}u(x,t)\psi_R(x)\dx +\int u(x,t)\nabla v(x,t)\cdot\nabla\psi_R(x)\dx\nonumber\\  
&=&-\int u(x,t)(-\Delta)^{\alpha/2}\psi_R(x)\dx \label{J}\\
&\qquad& +\frac12\iint u(x,t)u(y,t)\nabla K_\gamma(x-y)\cdot\big(\nabla\psi_R(x) -\nabla\psi_R(y)\big)\dy\dx,\nonumber
\eea
where we applied the formula $v=K_\gamma\ast u$, and the last expression follows by the symmetrization of the double integral: $x\mapsto y$, $y\mapsto x$. 
Since $u(x,t)\ge 0$, by the scaling relation \rf{scal} and Lemma \ref{Getoor}, we obtain 
\begin{equation}\label{J0}
-\int u(x,t)(-\Delta)^{\alpha/2}\psi_R(x)\dx\ge -R^{-\alpha} k_\alpha\int_{\{|x|\le R\}}u(x,t)\dx.
\end{equation}

Now, let $\e\in\left(0, \frac1{\sqrt3}\right)$. By Lemma \ref{prop1}, the weight function $\psi_R$ in \rf{bump} is concave for $|x|\le \e R$ with a concavity  constant  $\theta=\theta(\e).$ 
Thus, by Lemma~\ref{prop3}, we have 
\[
\nabla K_\gamma (x-y)\cdot \left(\nabla \psi_R(x)-\nabla \psi_R(y)\right)\geq R^{-2}\frac{\theta(\e)}{2\pi}g_{\gamma}(|x-y|)
\] 
for $|x|, |y|<\e R$.
Hence, the bilinear term on the right-hand side of \rf{J} satisfies 
\begin{equation}\label{bilinear}
\begin{split}
\frac12\iint u(x,t)&u(y,t)\nabla K_\gamma(x-y)\cdot\big(\nabla\psi_R(x) -\nabla\psi_R(y)\big)\dy\dx\\
\ge 
&R^{-2}\frac{\theta(\e)}{4\pi} \int_{\{|x|<\e{R}\}}\int_{\{|y|< \e{R}\}} g_\gamma(|x-y|)u(x,t)u(y,t)\dy\dx + \frac{1}{2}J,
\end{split}
\end{equation}
where the letter  $J$ denotes the integral 
$$
J=\iint_{{\mathbb R}^2\times\R^2\setminus\left(\{|x|<\e{R}\}\times\{|y|< \e{R}\}\right)}u(x,t)u(y,t)\nabla K_\gamma(x-y)\cdot (\nabla\psi_R(x)-\nabla\psi_R(y))\dy\dx.
$$ 
We estimate the first integral on the right-hand side of \eqref{bilinear} 
in the following way
\begin{equation}\label{JJJ}
\begin{split}
&\int_{\{|x|<\e{R} \}}\int_{\{|y|< \e{R} \}} g_\gamma(|x-y|)u(x,t)u(y,t)\dy\dx \\ 
&\ge g_\gamma(2\e R)\bigg(M_R(t)-\int_{\{\e R\leq |x|\le {R} \}}u(x,t)\dx\bigg)^2\\
&\ge  g_\gamma(2\e R)M_R^2(t)-2 g_\gamma(2\e R) M_R(t)\int_{\{\e R\leq |x|\le {R} \}}u(x,t)\frac{1-\psi_R(x)}{\inf_{\{|x|\ge \e{R} \}} \big(1-\psi_R(x)\big) }{\dx}\\
&\ge g_\gamma(2\e R) M_R(t)^2-2 C_\e M_R(t)(M-w_R(t)),
\end{split}
\end{equation}
where  $C_\e=\left(\inf_{\{|x|\ge \e{R} \}} \big(1-\psi_R(x)\big)\right)^{-1}= \big(1-(1-\e^2)^2\big)^{-1}$. In the above inequalities we used the fact that $g_\gamma$ is a continuous decreasing function and $0\leq g_\gamma \leq 1$.    
Next, 
since we have the inclusion  
\bea 
&&{\mathbb R}^2\times\R^2\setminus\Big(\{|x|<\e{R}\}\times\{|y|< \e{R}\}\Big)\subset  \nonumber\\
&&\Big(\{|x|<R\}\times\{|y|\ge\e{R}\}\Big) \cup \Big(\{|x|\ge \e{R} \}\times\{|y|<R\}\Big) \cup \Big(\{|x|\ge R\}\times\{|y|\ge R\}\Big)\nonumber
\eea
 and the factor with $\nabla\psi_R$ vanishes on the  set  $\{|x|\ge R\}\times\{|y|\ge R\}$, we obtain immediately the  estimate
\bea
|J|&\le& 2CR^{-2}\int_{\{|x|<R\}}\int_{\{|y|\ge \e{R} \}}u(x,t)u(y,t) \frac{1-\psi_R(y)}{\inf_{\{|y|\ge \e{R} \}}\big(1-\psi_R(y)\big)}{\dx\, \dy}\nonumber\\
&\le& 2R^{-2}C C_\e M_R(t) \int u(y,t)\big(1-\psi_R(y)\big)\dy\nonumber\\
&\le& 2R^{-2}C C_\e M_R(t)(M-w_R(t)), \label{JJ}
\eea
where $C=\sup|z\cdot\nabla K_\gamma(z)|\, \|D^2\psi\|_\infty$.
Finally, estimates \rf{J0}--\rf{JJ} as well as inequality \eqref{Mt} applied to equation \rf{J} lead to the inequalities
\begin{equation}\label{ineq}
\begin{split}
\frac{\rm d}{{\rm d}t}w_R(t)&\ge R^{-\alpha}M_R(t)\bigg(-k_\alpha + \frac{\theta(\e)}{4\pi}  R^{\alpha-2}g_\gamma(2\e R)M_R(t)+C(\e) R^{\alpha-2}(w_R(t)-M)\bigg)\\
&\ge R^{-\alpha}w_R(t)\bigg(-k_\alpha + \frac{\theta(\e)}{4\pi}  R^{\alpha-2}g_\gamma(2\e R)w_R(t)+C(\e) R^{\alpha-2}(w_R(t)-M)\bigg),
\end{split}
\end{equation}
whenever the expression in the parentheses is nonnegative, with $C(\e)=3CC_\e= 3C\big(1-(1-\e^2)^2\big)^{-1}$.

Now, notice that the linear function  of $w_R$ in the parentheses on the right-hand side of \eqref{ineq} is monotone increasing. Thus,
if at the initial moment $t=0$ the right-hand side of \rf{ineq} is positive,   then  
$w_R(t)$  will increase indefinitely in time. Consequently, after a moment $T={\mathcal O}\bigg( R^\alpha \big(\frac{M}{w_R(0)}-1\big)\bigg)$ the function $w_R(t)$ will become larger than the total mass $M$. This is  a~contradiction with the global-in-time existence of a nonnegative solution $u$ since it conserves mass \rf{M}.

Now, let us analyze the cases when the right-hand side of inequality \rf{ineq} is strictly positive.

 {\it Case (i)}:  $\alpha=2$ and $\gamma=0$. 
We recall that by Lemma \ref{Getoor} $k_2=8$ holds.
For $\gamma=0$, the Bessel potential $K_\gamma$ should be replaced 
by the fundamental solution $E_2(x)$ of Laplacian on $\R^2$ which satisfies $\nabla E_2(x)=-\frac{1}{2\pi}\frac{x}{|x|^2}$, so that 
$g_0(2\e R)=1$.
In view of \eqref{wazne}, the quantity  $\frac{\theta(\e)}{4\pi}$ is close to $\frac1{\pi} $ at the expense of taking sufficiently small $\e>0$. 
Choosing $\e>0$ small enough, we get blowup in the optimal range $M>8\pi$. Indeed, if $M>8\pi$, then there exists $\e>0$ small, and $R\ge R(\e)>0$ sufficiently large so that $w_R(0)$ is sufficiently close to $M$ and we have   
\[
-8+\frac1{\pi} w_R(0)+C(\e)(w_R(0)-M)>0.
\]
\medskip

 {\it Case (ii)}: $\alpha=2$, $\gamma>0$. If $M>8\pi$ and $u_0$ is sufficiently well concentrated near the origin, i.e. $\frac{\theta(\e)}{4\pi}g_\gamma(2\e R)w_R(0)>k_2=8$ and, at the same time $C(\e)(M-w_R(0))$ is sufficiently small, then the solution $u$ cannot be global-in-time. 
\medskip

 {\it Case (iii)}: 
In the case $\alpha<2$, the blowup occurs if for some $R>0$ the quantity \newline  $R^{\alpha-2}\int_{\{|x|<R\}}u_0(x)\dx$ is large enough, and simultaneously $u_0$ is well concentrated, i.e. $C(\e)(M-w_R(0))$ is small. 
 \qed

\section{Global existence of large mass solutions}

In this section we prove Theorem \ref{ex}. 
Our proof splits naturally into several parts. The first one is a construction of local-in-time   mild solutions with  initial data in $L^1$ with an estimate of the existence time {\em uniform} in $\gamma$.
The second step consists in proving the continuation of such a local solution to a global-in-time one satisfying the (nonoptimal) decay estimate $\limsup_{t\to\infty} t^{1/\sigma-1/p}\|u(t)\|_p<\infty$ for each fixed $p\in(4/3,2)$ and any  $\sigma\in(1,p)$. 
Finally, we  will prove a uniform global $L^1$ bound $\sup_{t>0}\|u(t)\|_1<\infty$
as well as  the optimal decay (hypercontractive) 
estimate $\sup_{t>0}t^{1-1/p}\|u(t)\|_p<\infty$.

First of all, the Cauchy problem \rf{equ}--\rf{ini} is studied via the integral equation (a.~k.~a. the Duhamel formula) 
\be
u(t)={\rm e}^{t\Delta}u_0+B(u,u)(t),\label{D}
\ee
whose solutions are called {\em mild}\,    solutions of the original Cauchy problem. 
Here, the bilinear term $B$ is defined as 
\be
B(u,z)(t)=-\int_0^t\left(\nabla{\rm e}^{(t-s)\Delta}\right)\cdot\left(u(s)\, \nabla(-\Delta+\gamma)^{-1}z(s)\right)\ds.\label{form}
\ee
Then, to solve equation \rf{D} in a Banach space $(\mathcal E,\xn\,.\,\xn)$ of vector-valued functions, it is sufficient to prove the boundedness of the bilinear form  $B:{\mathcal E}\times{\mathcal E}\to {\mathcal E}$ 
\be
\xn B(u,z)\xn\le \eta \xn u\xn \xn z\xn,\label{form2}
\ee
with a constant $\eta$  independent of $u$ and $z$.  
The first and the second steps toward the proof of Theorem \ref{ex} are based on a lemma which is  convenient to formulate in the following way:  

\begin{lemma}\label{fix}
If $\xn B(u,z)\xn\le \eta\xn u\xn\, \xn z \xn$ and 
$\xn {\rm e}^{t\Delta}u_0\xn \le R<\frac{1}{4\eta}$, then equation \rf{D} has a solution which is  unique in the ball of radius $2R$ in the space  $\mathcal E$. 
Moreover,  these solutions depend continuously on the initial data, i.e. $\xn u-\tilde u\xn\le C\xn {\rm e}^{t\Delta}(u_0-\tilde u_0)\xn$. 
\end{lemma}
The detailed proof of Lemma \ref{fix} can be found in \cite{Lem}, \cite{B-SM}.  The reasoning involves the Banach contraction theorem, the unique solution being achieved as a limit in ${\mathcal E}$ of the sequence of successive  approximations 
\be
w_0(t)={\rm e}^{t\Delta}u_0,\ \ w_{n+1}=w_0+B(w_n,w_n).\label{c_sequence}
\ee

\subsection*
{Step 1. Local-in-time solutions with the initial data in $L^1$}  

\begin{lemma}\label{lem:loc}
For every $u_0\in L^1(\R^2)$ and 
 $p\in\big(\frac43,2\big)$, there exists $T>0$ independent of $\gamma$,  such that  equation \eqref{D} has a solution $u=u(x,t)$ in the space 
\be
{\mathcal E}=\{u\in L^\infty_{\rm loc}((0,T); L^p(\mathbb R^2)):\ \ \sup_{0<t\le T}t^{1-1/p}\|u(t)\|_p<\infty\},\label{ET}
\ee
endowed with the norm 
\be
\xn u\xn\equiv \sup_{0<t\le T}t^{1-1/p}\|u(t)\|_p<\infty.\label{normT}
\ee 
\end{lemma}

\begin{proof}
Let $\frac1r=\frac2p-\frac12$, so that $r\in(1,2)$. Moreover, denote by $q$ a number satisfying $\frac1p +\frac1q=\frac1r$. We estimate  the bilinear form $B$ for each $t\in(0,T)$ using \rf{lin2} and \rf{Kgamma} 
\begin{equation} \nonumber
\begin{split}
\|B(u,z)(t)\| _p 
&\le C\int_0^t(t-s)^{-1/2+1/p-1/r}\|u(s)\nabla(-\Delta+\gamma)^{-1}z(s)\|_r\ds \\ 
&\le C\int_0^t(t-s)^{-1/2+1/p-2/p+1/2}\|u(s)\|_p\|\nabla(-\Delta+\gamma)^{-1}z(s)\|_q\ds\\ 
&\le  C\gamma^{-1/2-1/q+1/p}\int_0^t(t-s)^{-1/p} s^{2(1/p-1)}\left(\sup_{0<s\le t}s^{1-1/p}\|u(s)\|_p\right)\\
&\qquad  \times \,\left(\sup_{0<s\le t}s^{(1-1/p)}\|z(s)\|_p\right)\ds, \\
&\le t^{1/p-1}C\xn u\xn\, \xn z\xn
\end{split}
\end{equation}
with a constant $C>0$ {\em independent} of $\gamma>0$ (and   also of $T>0$) since $-\frac12-\frac1q+\frac1p=0$.
The last inequality is a consequence of the fact that
\[
\int_0^t(t-s)^{-1/p}s^{2(1/p-1)}\ds= Ct^{1/p-1}.
\]
Finally, given $u_0\in L^1(\mathbb R^2)$, by \rf{function}  we may choose $T>0$ so small to have  
$$\sup_{0<t\le T}t^{1-1/p}\|{\rm e}^{t\Delta}u_0\|_p<\frac{1}{4C}.$$ 
 Thus, the existence of a solution in $\mathcal E$ follows by an application of Lemma \ref{fix}. 
\end{proof}

 \begin{lemma}[The  $L^1$-bound]\label{L-1}
 The solution constructed in Lemma \ref{lem:loc} satisfies
 \be
 \sup_{0<t\le T}\|u(t)\|_1<\infty.\label{L1}
 \ee  
 \end{lemma}
 
\begin{proof}  
Let us take the sequence $w_n$ of approximations of $u$ as in \eqref{c_sequence}. By Lemma \ref{lem:loc} we know that
there exists a constant $C_0$ such that
\be
\xn w_n\xn \label{ogr}\le C_0<\infty.
\ee 
Next, we take any $r\in [\frac43, 2)$  such that $r<p$, and interpolate the $L^r$ norm 
\be
\left\|w_n\right\|_r\leq\left\|w_n\right\|_p^{1-\theta}\left\|w_n\right\|_1^{\theta},\label{interpolacja}
\ee
where $\theta=\frac{\frac1r-\frac1p}{1-\frac1p}$, and thus $1-\theta=\frac{1-\frac1r}{1-\frac1p}$. 
Next, observe that for $\frac1r+\frac1q=1$, by \rf{c_sequence} and \rf{lin2} we have 
\bea\label{osz3}
\left\|w_{n+1}(t)\right\|_1& \leq &\left\|u_0\right\|_1+C\int_0^t (t-s)^{-1/2}\left\|w_n(s)\nabla (-\Delta+\gamma)^{-1}w_n(s)\right\|_1\ds\\ \nonumber
& \leq &\left\|u_0\right\|_1+C\int_0^t (t-s)^{-1/2}\left\|w_n(s)\right\|_r\left\|\nabla(-\Delta+\gamma)^{-1}w_n(s) \right\|_q\ds\\ \nonumber
& \leq &\left\|u_0\right\|_1+C\int_0^t (t-s)^{-1/2}\left\|w_n(s)\right\|_r^2\ds,\\ \nonumber
& \leq &\left\|u_0\right\|_1+C\int_0^t (t-s)^{-1/2}\left\|w_n(s)\right\|_p^{2(1-\theta)}\left\|w_n(s)\right\|_1^{2\theta}\ds\\ \nonumber
&\le &\|u_0\|_1+C\int_0^t(t-s)^{-1/2}s^{2(1/r-1)}\|w_n(s)\|_1^{2\theta}\ds, 
\eea
where we used inequality \eqref{Kgamma} for $r\geq\frac43$ and \eqref{ogr}.  
Now, for $t\le T$ we take $r=\frac43$ so that $q=4$ and note that $\int_0^t(t-s)^{-1/2}s^{-1/2}\ds=\pi=\,{\rm const}$. 
Let us define 
$$
A_0(t)=\sup_{0<s\le t}\|w_0(s)\|_1,\quad \dots,\quad
A_n(t)=\sup_{0<s\le t}\|w_n(s)\|_1,
$$
so that  $A_0(t)<\infty$, and as a consequence of \rf{osz3} we arrive at
\[
A_{n+1}(t)\le C_1+C_2A_n(t)^\varrho<\infty,
\]
where $\varrho=2\theta<1$, since $r=4/3$ and $p<2$.
Therefore $A_n(t)$ is uniformly bounded in $L^\infty(0,T)$ (with a bound which depends on  $C_1$ and $C_2$ but is  independent of $n$)  by an easy recurrence argument. From the fact that $\xn w_n-u\xn\rightarrow 0$ when $n\rightarrow \infty$ we infer that for any $t>0$ the family $w_n(\cdot,t)$ converges to $u(\cdot,t)$ in $L^1_{\rm loc}(\R^2)$. 
Applying the  Fatou lemma we see that 
$
\left\|u(t)\right\|_1\leq \liminf_{n}A_n(t)\leq C
$  
for every fixed   $0<t\le T$. 
\end{proof}

 \subsection*{Step 2. Global-in-time solutions for $\gamma$ large}
 
 \begin{lemma}\label{lem:nonopt}
Now we keep $p\in\big(\frac43,2\big)$  and take any $\sigma\in(1,p)$. For every $u_0\in L^\sigma(\R^2)$, there exists a constant $\gamma(u_0)>0$ such that for all $\gamma\geq \gamma(u_0)$ equation \eqref{D} has a unique solution in the
 new functional space 
\be
\widetilde{\mathcal E}=\{u\in L^\infty_{\rm loc}((0,\infty); L^p(\mathbb R^2)):\ \ \sup_{t>0}t^{1/\sigma-1/p}\|u(t)\|_p<\infty\},\label{E}
\ee
endowed with the norm 
\be
\xn u\xn\equiv \sup_{t>0}t^{1/\sigma-1/p}\|u(t)\|_p<\infty.\label{norm}
\ee 
\end{lemma}

\begin{proof}
Let $\frac1r=\frac2p+\frac12-\frac1\sigma$ for some suitable $\sigma\in(1,p)$ so that $r\in(1,2)$. Moreover, denote by $q$ a number satisfying $\frac1p +\frac1q=\frac1r$. Under this choice of parameters we make sure that $\frac1p-\frac1q<\frac12$ and $q>2$, so that we can use \rf{Kgamma}  to estimate  the bilinear form $B$  
\begin{equation} \nonumber
\begin{split}
&\|B(u,z)(t)\| _p 
\le C\int_0^t(t-s)^{-1/2+1/p-1/r}\|u(s)\nabla(-\Delta+\gamma)^{-1}z(s)\|_r\ds \\ 
&\le C\int_0^t(t-s)^{2/\sigma-1/p-1}\|u(s)\|_p\|\nabla(-\Delta+\gamma)^{-1}z(s)\|_q\ds\\ 
&\le  C\gamma^{-1/2-1/q+1/p}\int_0^t(t-s)^{1/\sigma-1/p-1} s^{2/p-2/\sigma}\\
&\times\left(\sup_{0<s\le t}s^{1/\sigma-1/p}\|u(s)\|_p\right)
\left(\sup_{0<s\le t}s^{1/\sigma-1/p}\|z(s)\|_p\right)\ds \\ 
&\le t^{1/p-1/\sigma}C\gamma^{1/\sigma-1}\xn u\xn\, \xn z\xn.
\end{split}
\end{equation}
Thus, we obtained inequality \rf{form2} 
with the norm defined in \eqref{norm} and 
with $\eta=C\gamma^{1/\sigma-1}$. 
We may choose $\gamma(u_0)$ so large to have 
$$\sup_{t>0} t^{1/\sigma-1/p}\|e^{t\Delta}u_0\|_p<\gamma^{1-1/\sigma}/4C$$
for all $\gamma\ge\gamma(u_0)$, which is possible due to estimate 
\eqref{lin1}. Then, the proof is completed by applying Lemma \ref{fix}.
\end{proof}

\subsection*{ Step 3. Proof of Theorem \ref{ex} and optimal hypercontractive estimates}
\ 

By Lemma \ref{lem:loc}, we have a local-in-time solution on an interval $(0,T]$ with $T>0$ independent of $\gamma$.
Moreover, 
 $u\left(\frac{T}{2}\right)\in L^\sigma(\R^2)\subset L^1\cap L^p$.
Thus, we may continue 
this  local-in-time solution $u(t)$ to the whole half-line $(0,\infty)$ choosing $\gamma>0$ sufficiently large.   We notice that on the interval $\left(\frac{T}2,T\right)$ solutions obtained in Lemma \ref{lem:loc} and Lemma \ref{lem:nonopt} coincide as a consequence of uniqueness assertion of Lemma \ref{lem:loc}.
It remains to prove optimal decay estimates.
\medskip 

{\it The optimal $L^1$-bound.} 

 Next, we show that the  solution satisfies the uniform global estimate 
 \be
 \sup_{t>0}\|u(t)\|_1<\infty.\label{ell-1}
 \ee
For $t\ge T$, similarly to the proof of Lemma \ref{L-1},  we consider a sequence of approximations of $u$, combine  estimates \rf{interpolacja} and an analog of \rf{osz3} (this time on $(T,\infty)$) with   $1<\sigma<\frac43<r<p<2<q$,  $\theta=\frac{\frac1r-\frac1p}{1-\frac1p}$. Moreover, we define
\be
\e=\frac{1-\frac1\sigma}{1-\frac1p}\;, \;\mbox{so that}\; t^{1/\sigma-1/p}=t^{(1-1/p)(1-\e)}. \label{wybor_e}
\ee
Again, we arrive at an estimate similar to \eqref{L1}
\[
\left\|w_{n+1}(t)\right\|_1\leq \left\|u_0\right\|_1+\int_0^t(t-s)^{-1/2}\left\|w_n(s)\right\|_p^{2(1-\theta)}\left\|w_n(s)\right\|_1^{2\theta}\ds,
\] 
where $\varrho=2\theta<1$ since we can choose $\frac2r<1+\frac1p$. We split the integral on the right-hand side into two integrals, over the interval $\left(0,\frac{T}2\right)$ and the integral over $\left(\frac{T}2,t\right)$. The first one is estimated by $C(T)$ in view of Lemma \ref{lem:loc} and \eqref{L1}. To estimate the second interval, we notice that 
\begin{equation} \nonumber
\begin{split}
&\int_{T/2}^t(t-s)^{-1/2}\left\|w_n(s)\right\|_p^{2(1-\theta)}\left\|w_n(s)\right\|_1^{2\theta}\ds\\
&\stackrel{\eqref{norm}}{\leq} C A_n(t)^{\varrho}\int_{T/2}^t(t-s)^{-1/2}\left(s-\frac{T}2\right)^{2(1-\theta)(1/p-1/\sigma)}\ds,
\end{split}
\end{equation}
but 
\[
\int_{T/2}^t(t-s)^{-1/2}\left(s-\frac{T}2\right)^{2(1-\theta)(1/p-1/\sigma)}\ds\leq C\left(t-\frac{T}2\right)^{1/2-2\left(1-1/r\right)(1-\e)}
\] 
holds with $\e$ as in \eqref{wybor_e}. We notice that choosing $\sigma>1$ close enough to $1$ and $r>\frac43$ close enough to $\frac43$ we ensure that  
$1/2-2(1-1/r)(1-\e)<0$. We proceed further  as  in the proof of Lemma \ref{L-1} and arrive at \eqref{ell-1}.

\begin{remark}
One can show by a standard method that 
$u\in {\mathcal C}([0,T);L^1(\mathbb R^2))$. Here, it suffices to use the continuity of the bilinear form $B$ as, {\it e.g.}, in  the proof in  \cite[Theorem~1.1]{BGK}. 
\end{remark}

 {\it The optimal hypercontractive estimate for $p>1$. }
 
First, we improve the decay estimates from Lemma \ref{lem:nonopt}.

\begin{lemma}\label{step3} 
For each $p\in(\frac43,2)$, the  solution of \eqref{equ}--\eqref{eqv} with $u_0\in L^1(\R^2)$ satisfies
\be
\sup_{t>0}t^{1-1/p}\left\|u(\cdot,t)\right\|_p<\infty.\label{osz6}
\ee
\end{lemma}
\noindent
\proof \ \  
By  definition \eqref{E} of the space $\widetilde{\mathcal E}$, one immediately sees that it is enough to prove \eqref{osz6} for $t\geq T$. 
By the Duhamel  formula \rf{D}, \rf{lin2} and \rf{Kgamma}, we have
\begin{equation}\label{osz7}
\begin{split}
&\left\|u(t)\right\|_p \leq \left\|e^{t\Delta}u_0\right\|_p+C\int_0^t (t-s)^{-1/2-1+1/p}\left\|u(s)\nabla  (-\Delta+\gamma)^{-1}u(s)\right\|_1\ds\\ 
\leq &\left\|e^{t\Delta}u_0\right\|_p+C\int_0^t (t-s)^{-1/2-1+1/p}\left\|u(s)\right\|_p\left\|\nabla  (-\Delta+\gamma)^{-1}u(s)\right\|_q\ds\\
\leq &Ct^{1/p-1}\left\|u_0\right\|_1+C\int_0^{T/2} (t-s)^{-3/2+1/p}\left\|u(s)\right\|_p^2\ds+C\int_{T/2}^t (t-s)^{-3/2+1/p}\left\|u(s)\right\|_p^2\ds.\\ 
\end{split}
\end{equation}
Using \eqref{E} we estimate the second term on the right-hand side of \eqref{osz7} as
\[
\int_0^{T/2} (t-s)^{-3/2+1/p}\left\|u(s)\right\|_p^2\ds
\leq C\int_0^{T/2} (t-s)^{-3/2+1/p}s^{2(1/p-1)}\ds\leq C(T)\left(t-\frac{T}2\right)^{-3/2+1/p}.
\]
Hence for $t\geq T$ and in view of the fact that 
\be
\mbox{for}\;\; t\geq T \;\mbox{it holds}\;\;\frac{t}{2}\leq t-\frac{T}2\label{tt}
\ee
relation \eqref{osz7}  reads 
\[
t^{1-1/p}\left\|u(t)\right\|_p\leq C\left\|u_0\right\|_1+Ct^{-1/2}+Ct^{1-1/p}\int_{T/2}^t (t-s)^{-3/2+1/p}\left\|u(s)\right\|_p^2\ds. 
\]
In turn, in view of \eqref{E} and owing to definition of $\e$ in \eqref{wybor_e}, for $t\geq T$ we arrive at
\be
t^{1-1/p}\left\|u(t)\right\|_p\leq C\left\|u_0\right\|_1+C+Ct^{1-1/p}\int_{T/2}^t (t-s)^{-3/2+1/p}\left(s-\frac{T}2\right)^{-2(1-1/p)(1-\e)}\ds.\label{osz8}
\ee
Next we use the inequality  
\[
\int_{T/2}^t (t-s)^{-3/2+1/p}\left(s-\frac{T}2\right)^{-2(1-1/p)(1-\e)}\ds\leq C\left(t-\frac{T}2\right)^{1/p-1/2-2(1-1/p)(1-\e)}, 
\]
to see that by \eqref{tt} for $t\geq T$ \eqref{osz8} yields
\[
t^{1-1/p}\left\|u(t)\right\|_p\leq C+Ct^{1/2}\left(t-\frac{T}2\right)^{-2(1-1/p)(1-\e)}\leq C+Ct^{1/2-2(1-1/p)(1-\e)}.
\]
Since $p\in(\frac43,2)$, it is enough to pick up $\e>0$ small enough to ensure that $\frac12-2\left(1-\frac1p\right)(1-\e)<0$, and we arrive at
\eqref{osz6} for $t\geq T$, Lemma \ref{step3} is proved. 
\qed
\bigskip

 {\it The optimal decay  estimate for other $p\in\left(1,\infty\right)$.}
 
 In view of Lemma \ref{step3}, we see that Theorem \ref{ex} is true for $p=1$ and $p\in(\frac43,2)$. Since for $p\in (1, \frac43]$, we have by interpolation 
\[
t^{1-1/p}\left\|u(t)\right\|_p\leq \left\|u(t)\right\|_1^\vartheta \left(t^{1-1/\bar{p}}\left\|u(t)\right\|_{\bar{p}}\right)^{1-\vartheta},
\]
where $\bar{p}<2$,  $\vartheta=\frac{\frac1p-\frac{1}{\bar{p}}}{1-\frac{1}{\bar{p}}}$, and therefore Theorem \ref{ex} holds also for $p\in (1,\frac43]$.

Now, we can interpolate  estimate \rf{osz6} for $p\in \left(\frac43,2\right)$ and \rf{ell-1} to get \rf{osz6} with any $p\in[1,2)$.  

The last step of the proof of Theorem \ref{ex} is the extrapolation of the hypercontractive estimates \rf{osz6} for $q\in [2, \infty)$. Actually, it is enough to obtain the decay estimate for $q\in (2, \infty)$, the remaining case $q=2$ will follow by simple interpolation. 
Taking $q\in (2,\infty)$, we know that
\begin{equation}\label{osz9}
\begin{split}
&\left\|u(t)\right\|_q \leq Ct^{1/q-1}\left\|u_0\right\|_1+C\int_0^t (t-s)^{-1/2-1/r+1/q}\left\|u(s)\nabla  (-\Delta+\gamma)^{-1}u(s)\right\|_r\ds\\ 
\leq &Ct^{1/q-1}\left\|u_0\right\|_1+C\int_0^t (t-s)^{-1/2-1/r+1/q}\left\|u(s)\right\|_\sigma\left\|\nabla  (-\Delta+\gamma)^{-1}u(s)\right\|_\rho \ds\\
\leq &Ct^{1/q-1}\left\|u_0\right\|_1+C\int_0^{t} (t-s)^{-1/2-1/r+1/q}\left\|u(s)\right\|_\sigma^2\ds.\\ 
\end{split}
\end{equation}
Here $r$ is chosen in such a way that $\frac1r=\frac{2}{\sigma}-\frac12$, so that for $\sigma\in (\frac43,2)$ we have $r\in(1,2)$, $r$ close to $2$. At the same time $\frac1\rho+\frac1\sigma=\frac1r$ and $\frac1\rho=\frac1\sigma-\frac12$, the above choice of parameters allows us to apply \eqref{Kgamma} to \eqref{osz9}. Relation \eqref{osz6} with $\sigma\in\left(\frac43,2\right)$
$\left\|u(s)\right\|_\sigma\leq Cs^{1/\sigma -1}$,
 applied to \eqref{osz9} yields 
\[
\left\|u(t)\right\|_q\leq Ct^{1/q-1}\left\|u_0\right\|_1+C\int_0^{t} (t-s)^{-1/2-1/r+1/q}s^{2\left(1/\sigma -1\right)}\ds.
\]
Since 
\[
\int_0^{t} (t-s)^{-1/2-1/r+1/q}s^{2\left(1/\sigma -1\right)}\ds= Ct^{-1/2-1/r+1/q+2/\sigma-1},      
\]
we notice that 
$\left\|u(t)\right\|_q\leq Ct^{1/q-1}\left\|u_0\right\|_1+Ct^{1/q-1}$ 
holds in view of the relation 
\[
-\frac12-\frac1r+\frac1q+\frac2\sigma-1=\frac1q-1.
\]
Thus,  the decay estimate for $q>2$ is proved. 
 \qed
\bigskip

\section*{Acknowledgments}
This work was initiated during the visits of T. Cie\'slak at Uniwersytet Wroc{\l}awski, T.C. would like to express his gratitude for support and hospitality. 
The first,  the third and the fourth authors were supported by the NCN grant  2013/09/B/ST1/04412. 
The second author was partially supported by the Polish Ministry of Science and Higher Education under the Iuventus Plus grant No. 0073/IP3/2011/71.  The fourth author was also supported by the grant  DEC-2012/05/B/ST1/00692.

\end{document}